\documentclass{article}
\usepackage{mathrsfs}
\usepackage{amsfonts}
\usepackage{tikz}
\usepackage{amssymb,bm}
\usepackage{amsmath}
\usepackage{amsthm}
\usepackage{bbding}
\usepackage{longtable,booktabs}
\usetikzlibrary{matrix,shapes,snakes}
\allowdisplaybreaks
\usepackage{caption}
\usepackage[titletoc]{appendix}

\newcommand{\C}{\mathbb C}

\newcommand{\Q}{\mathbb Q}
\newcommand{\Z}{\mathbb Z}

\newtheorem{thm}{Theorem}[section]
\newtheorem{Lemma}{Lemma}[section]

\newtheorem{Remark}{Remark}[section]
\newtheorem{Example}{Example}[section]
\newtheorem{Corollary}{Corollary}[section]

\begin{document}
\title{$\Z_q$-valued generalized bent functions in odd characteristics}

\author{Libo Wang$^1$,  Baofeng Wu$^2$, Zhuojun Liu$^1$}
 \date{{\small $^1$ KLMM, Academy of Mathematics and Systems Science, Chinese Academy of Sciences,
 Beijing 100190,  China}\\
 {\small $^2$ SKLOIS, Institute of Information Engineering, Chinese Academy of Sciences, Beijing 100093, China}}

\maketitle
\normalsize
\begin{abstract}
In this paper, we  investigate properties of functions from  $\Z_{p}^n$ to $\Z_q$, where $p$ is an odd prime  and $q$ is a positive integer
divided by $p$.  we present the sufficient and necessary conditions for bent-ness of such generalized Boolean functions in terms of classical $p$-ary bent functions,
when $q=p^k$. When $q$ is divided by $p$ but not a power of it, we give an sufficient condition for weakly regular gbent functions.
Some related constructions are also obtained.

\end{abstract}

\section{Introduction}
\label{sec1}
As an interesting combinatorial object with the maximum Hamming distance to the set of all affine Boolean  functions,
bent functions were introduced by Rothaus \cite{Rothaus} in 1976. Such functions have been extensively
studied because of their important applications in cryptograph \cite{Carlet}, sequence design \cite{Olsen},
coding theory \cite{Canteaut}, and association schemes \cite{Pott}.  We refer to \cite{Cusick, Tokareva} for more on
cryptographic Boolean functions and bent functions.

In \cite{Schmidt}, Schmidt proposed the generalized bent  functions from
$\Z_2^n$ to $\Z_4$, which can be used to constant amplitude codes and $\Z_4$-linear codes for CDMA communications.
Later,  generalized bent  functions from $\Z_2^n$ to $\Z_8$ and $\Z_{16}$ were studied in \cite{Stanica} and \cite{Martinsen1},
respectively.  Recently, a generalization of bent functions from $\Z_{2}^n$ to $\Z_q$, where $q \geq 2$
is any positive even integer, have attracted  more and more attention. Existence, characterizations and constructions of them were studied by several authors \cite{liuhaiying,Hodzic,Martinsen,Tang}.
In fact, the study of generalizations of Boolean bent functions started as early as in 1985 . In \cite{Kumar} Kumar introduced the notion of generalized  bent functions from $\Z_q^n$ to $\Z_q$,
where $q$ is any positive integer. Such functions are often called  $p$-ary generalized bent functions when $q=p$ is an odd prime number.
As for generalized bent functions from $\Z_2^n$ to $\Z_q$, a natural problem is to generalize them to the $p$-ary case, i.e., studying generalized bent functions from $\Z_p^n$ to $\Z_q$ for an odd prime $p$. To the best of  our knowledge, there is no literature  discussing  such  generalized bent functions up to present. We mainly focus on the case when  $q$ is
divided by $p$ in the present paper.

Throughout this paper, let $\Z_{p}^n$ be an $n$-dimensional vector space over $\Z_p$   for an odd prime $p$,
$\Z_{p}$ (or $\Z_q$) be the ring of integer modulo $p$ (or $q$), and $\C$ be the field of complex numbers,
where $n, q$ are positive integers. If $\mathbf{x}=(x_1,...,x_n)$ and $\mathbf{y}=(y_1,...,y_n)$ are two
vectors in $\Z_{p}^n$, we define the  inner product by $\mathbf{x}\cdot\mathbf{y} = x_1y_1+ \cdots + x_ny_n$ (mod $p$) (without cause of confusion,  we always omit ``mod $p$" in the sequel). For a complex number $z=a+b\sqrt{-1}$, the absolute  value of
$z$ is $|z|=\sqrt{a^2+b^2}$ and $\bar{z}=a-b\sqrt{-1}$ denotes the complex conjugate of $z$, where $a$ and $b$ are real numbers.

A function from $\Z_p^n$ to $\Z_q$ is called a generalized  Boolean function on $n$ variables, whose set is denoted
by $\mathcal{GB}_n^q$. We emphasize that $q$ is a positive integer divided by $p$ in this paper. For  a function $f\in\mathcal{GB}_n^q$, the generalized Walsh-Hadamard transform, which is
a function $\mathcal{H}_f: \Z_p^n \rightarrow \C$, can be defined by
\begin{equation}
\label{01}
\mathcal{H}_f(\mathbf{u})= p^{-\frac{n}{2}} \sum \limits_{\mathbf{x} \in \Z_p^n} \zeta_p^{-\mathbf{u}\cdot\mathbf{x}} \zeta_{q}^{f(\mathbf{x})},
\end{equation}
for any $\mathbf{u} \in \Z_p^n$, where  $\zeta_p=e^{\frac{2 \pi \sqrt{-1}}{p}}$ and
$\zeta_q=e^{\frac{2 \pi \sqrt{-1}}{q}}$ represent the complex $p$-th and $q$-th primitive  roots of unity, respectively.
The inverse generalized Walsh-Hadamard transform of $f$ is
\begin{equation}
\label{02}
\zeta_q^{f(\mathbf{x})}= p^{-\frac{n}{2}} \sum \limits_{\mathbf{u} \in \Z_p^n} \zeta_p^{\mathbf{u}\cdot\mathbf{x}} \mathcal{H}_f(\mathbf{u}).
\end{equation}
We call the function $f$  a $\Z_q$-valued $p$-ary generalized  bent (gbent) function if $|\mathcal{H}_f(\mathbf{u})| =1$ for all $\mathbf{u} \in \Z_{p}^n$.
A gbent function $f$ is regular if there exists some generalized  Boolean  function $f^\ast$ satisfying
$\mathcal{H}_f(\mathbf{u})= \zeta_{q}^{f^\ast(\mathbf{u})}$ for any $\mathbf{u} \in \Z_p^n$. Such a function $f^\ast$ is called the dual of $f$.
From the inverse generalized Walsh-Hadamard transform, it is easy to see that the dual $f^\ast$ of a gbent function $f$ is also regular.
A gbent function $f$ is called weakly regular if there exists some generalized  Boolean  function $f^\ast$  and a complex $\alpha$
with unit magnitude satisfying
$\mathcal{H}_f(\mathbf{u})=\alpha \zeta_{q}^{f^\ast(\mathbf{u})}$ for any $\mathbf{u} \in \Z_p^n$. When $q=p$, the gbent function $f$ is just a
classical $p$-ary bent function.

The rest of the paper is organized as follows. In Section \ref{sec2}, from the theory of cyclotomic fields we prove that for gbent functions
$f \in \mathcal{GB}_n^{p^k}$,  $\mathcal{H}_f(\textbf{u}) = \alpha \zeta_{p^k}^{f^{\ast}(\mathbf{u})}$, where $ \alpha \in 　\{\pm1, \pm \sqrt{-1} \}$.
In  Section \ref{sec3},  we give the sufficient and necessary  conditions for gbent functions $f \in \mathcal{GB}_n^{p^k}$, and the sufficient conditions
for weakly regular gbent functions  in $ \mathcal{GB}_n^{q}$ when $q$ is divided by $p$ but not a power of $p$. In Section \ref{sec4},
we give some related constructions. Concluding remarks are given in Section \ref{sec5}.

\section{Preliminaries}
\label{sec2}
In this section we will give some results on cyclotomic fields, which will be used in the following sections. We also completely determine
that what  Walsh spectra values can attain, for any gbent function in $\mathcal{GB}_n^{p^k}$, where $k$ is a positive integer.
Firstly, we state some basic facts  on the cyclotomic fields $K=\Q(\zeta_{p^k})$, which can be found in any book on algebraic number theory,
for example \cite{Washington}.

Let $\mathcal{O}_K$ be the ring of integers of $K=\Q(\zeta_{p^k})$. It is well known that $\mathcal{O}_K=\Z[\zeta_{p^k}]$. Any nonzero ideal $A$
of $\mathcal{O}_K$ can be uniquely (up to the order) expressed as
\[A=P_1^{a_1} \cdots P_s^{a_s},\]
where $P_1, \cdots, P_s$ are distinct prime ideals of $\mathcal{O}_K$ and $a_i \geq 1$, for $1\leq i \leq s$. In other words, the set $S(K)$ of
all the nonzero ideals of $\mathcal{O}_K$ is a free multiplicative  communicative semigroup with a basis $B(K)$, the set of all nonzero  prime ideals
of $\mathcal{O}_K$. Such semigroup $S(K)$ can be extended to the commutative group $I(K)$, called the group of fractional ideals of $K$. Each element of $I(K)$,
called a fractional ideals, has the form $AB^{-1}$, where $A,B$ are ideals of $\mathcal{O}_K$. For each $\alpha \in K^{\ast}=K \setminus \{0\}$, $\alpha \mathcal{O}_K$
is a fractional ideals, called a principle fractional ideals, and we have $(\alpha \mathcal{O}_K)(\beta \mathcal{O}_K)=\alpha \beta \mathcal{O}_K$,
$(\alpha \mathcal{O}_K)^{-1}=(\alpha^{-1}) \mathcal{O}_K$. Therefore, the set $P(K)$ of all principle fractional ideals is a subgroup of $I(K)$.
Some results on $K$ are given in the following lemmas.

\begin{Lemma}
\label{Lemma2.1}
Let $k\geq 2$ and $K=\Q(\zeta_{p^k})$, then

$(i)$ The field extension $K/\Q$ is Galois of degree $(p-1)p^{k-1}$ and the Galois group  $Gal(K/ \Q)=\{ \sigma_j | ~ j \in \Z, (j,p)=1 \}$,
where the automorphism $\sigma_j$ of $K$ is defined by $\zeta_{p^k} \mapsto \zeta_{p^k}^j$.

$(ii)$  The ring of integers in $K$ is $\mathcal{O}_K = \Z [\zeta_{p^k}]$ and $\{\zeta_{p^k}^j|~0 \leq j \leq (p-1)p^{k-1}-1\}$ is an integral basis of $\mathcal{O}_K$.
The group of roots of unity in $\mathcal{O}_K$ is $W_K=\{\zeta_{{2p}^k}^j|~0 \leq j \leq 2 p^k-1\}$

$(iii)$  The principle ideal $(1-\zeta_{p^k}) \mathcal{O}_K$ is a prime ideal of $\mathcal{O}_K$ and the rational prime $p$ is totally ramified in $\mathcal{O}_K$,
i.e., $p \mathcal{O}_K=( (1-\zeta_{p^k})\mathcal{O}_K )^{(p-1)p^{k-1}}$.
\end{Lemma}

We know that an algebraic integer $x$ of an algebraic number field is a root of unity if and only if all its conjugates $\sigma(x)$ (where $\sigma \in Gal(K/ \Q)$ )
have unit magnitude. We use $\sigma^{\ast}$ to denote the automorphism performing complex conjugation, i.e., $\sigma^{\ast}(\zeta_{p^k})=\zeta_{p^k}^{-1}$.
For the cyclotomic fields $K=\Q(\zeta_{p^k})$, the  Galois group $Gal(K/ \Q)$ is an Abelian group,  which ensures that
\[x\sigma^{\ast}(x)=1 \Rightarrow \sigma(x) \cdot \sigma^{\ast}(\sigma(x))=1, ~~ \forall~\sigma \in  Gal(K/ \Q),  \]
and we therefore have

\begin{Lemma}
\label{Lemma2.2}
If $x \in K=\Q(\zeta_{p^k})$ is an algebraic integer having unit magnitude, then $x$ is a root of unity.
\end{Lemma}

For a gbent function  $f \in \mathcal{GB}_n^{p^k}$, we will determine what the Walsh spectra values of $f$ can attain.
Firstly, we give the following lemma, which will be used to determine the Walsh spectra values of gbent functions in $\mathcal{GB}_n^{p^k}$.

\begin{Lemma}
[see \cite{Kumar}]
\label{Lemma2.3}
For a positive integer $q$,

$(i)$  If $q\equiv 0,1~~(mod~4)$, then $\sqrt{q} \in \Q(\zeta_q)$,

$(ii)$  If $q\equiv 2,3~~(mod~4)$, then $\sqrt{q} \in \Q(\zeta_{4q}) \backslash \Q(\zeta_{2q})$.
\end{Lemma}

\begin{Lemma}
\label{Lemma2.4}
For $f \in \mathcal{GB}_n^{p^k}$, i.e., $f: \Z_{p}^n \rightarrow \Z_{p^k}$,  $\mathcal{H}_f(\mathbf{u})$ is a root of unity for any
$\mathbf{u} \in \Z_p^n.$
\end{Lemma}

\begin{proof}
Since $f$ is gbent in $\mathcal{GB}_n^{p^k}$, then
\[(p^{\frac{n}{2}}\mathcal{H}_f(\mathbf{u})) \sigma^{\ast}(p^{\frac{n}{2}}\mathcal{H}_f(\mathbf{u}))= p^{\frac{n}{2}}\mathcal{H}_f(u)  \overline{p^{\frac{n}{2}}\mathcal{H}_f(u)} = p^n.\]
From  condition (iii) of Lemma \ref{Lemma2.1}, we have
\[((p^{\frac{n}{2}}\mathcal{H}_f(\mathbf{u})) \mathcal{O}_K)\sigma^{\ast}((p^{\frac{n}{2}}\mathcal{H}_f(\mathbf{u})) \mathcal{O}_K) = ((1-\zeta_{p^k})\mathcal{O}_K )^{n(p-1)p^{k-1}}. \]
According to the uniqueness of decomposition of $((p^{\frac{n}{2}}\mathcal{H}_f(\mathbf{u})) \mathcal{O}_K)\sigma^{\ast}((p^{\frac{n}{2}}\mathcal{H}_f(u)) \mathcal{O}_K)$,
we can assume that
\[(p^{\frac{n}{2}}\mathcal{H}_f(\mathbf{u})) \mathcal{O}_K= ((1-\zeta_{p^k})\mathcal{O}_K )^l,\]
for some $1 \leq l \leq n(p-1)p^{k-1}$. Therefore, we can get
\[\sigma^{\ast}((p^{\frac{n}{2}}\mathcal{H}_f(\mathbf{u})) \mathcal{O}_K) = ((1-\zeta_{p^k}^{p^k-1})\mathcal{O}_K )^l.\]
The algebraic integer $1-\zeta_{p^k}$ and $1-\zeta_{p^k}^{p^k-1}$ generate the same ideal in $\mathcal{O}_K$ as each is a unit times the other. So we have
\[(p^{\frac{n}{2}}\mathcal{H}_f(\mathbf{u})) \mathcal{O}_K = \sigma^{\ast}((p^{\frac{n}{2}}\mathcal{H}_f(\mathbf{u})) \mathcal{O}_K) = ((1-\zeta_{p^k}^{p^k-1})\mathcal{O}_K )^{\frac{p-1}{2}np^{k-1}}.\]
Further, we have
\[(p^{\frac{n}{2}}\mathcal{H}_f(\mathbf{u})\mathcal{O}_K)^2=p^n\mathcal{O}_K.\]
Therefore, there exists a unit $u \in \mathcal{O}_K$  such that $(p^{\frac{n}{2}}\mathcal{H}_f(\mathbf{u}))^2/p^n=u$, i.e., ${\mathcal{H}}^2_f(\mathbf{u})=u$,
so the Fourier coefficient $\mathcal{H}_f(\mathbf{u})$ is a algebraic integer. As $f$ is gbent, this algebraic integer has unit magnitude and must be a unity root, by Lemma \ref{Lemma2.2}.
\end{proof}

\begin{Lemma}
\label{Lemma2.5}
For gbent function $f \in \mathcal{GB}_n^{p^k}$,  there exists a function $f^{\ast}: \Z_p^{n}\rightarrow \Z_{p^k}$ such that
\begin{eqnarray*}
\begin{split}
 \mathcal{H}_f(\mathbf{u})=\left\{
\begin{array}{l}
\pm \zeta_{p^k}^{f^{\ast}(\mathbf{u})}  ~~~~~~~~~~~~$if$~~$n$~~is~~even~~or~~$n$~~is ~~odd~~and~~p \equiv 1(mod~4),\\
\pm \sqrt{-1} \zeta_{p^k}^{f^{\ast}(\mathbf{u})} ~~~~ $if$~~ $n$~~is ~~odd~~and~~p \equiv 3(mod~4).
\end{array}
\right.
\end{split}
\end{eqnarray*}
\end{Lemma}
\begin{proof}
(1) When $n$ is even, $p^{-\frac{n}{2}} \in \Q$, So $\mathcal{H}_f(\mathbf{u}) \in \Q(\zeta_{p^k})$, for any $\mathbf{u} \in \Z_p^n$.
From  Lemma \ref{Lemma2.4} and condition (ii) of Lemma \ref{Lemma2.1}, we have $\mathcal{H}_f(\mathbf{u}) = \zeta_{2p^k}^l$, for some
$0 \leq l \leq 2p^k-1$.

If $l$ is even, then set $f^{\ast}(\mathbf{u})=\frac{l}{2}$, we have $\mathcal{H}_f(\mathbf{u})=\zeta_{p^k}^{f^{\ast}(\mathbf{u})}$.

If $l$ is odd, then set $f^{\ast}(\mathbf{u})=\frac{p^k+l}{2}~(mod~p^k)$, we have $\mathcal{H}_f(\mathbf{u})=- \zeta_{p^k}^{f^{\ast}(\mathbf{u})}$.

(2) When $n$ is odd and $p\equiv 1 ~(mod~4)$, from Lemma \ref{Lemma2.3}, we have $\sqrt{p} \in \Q(\zeta_{p})  \subseteq \Q(\zeta_{p^k})$,
So $\mathcal{H}_f(\mathbf{u}) \in \Q(\zeta_{p^k})$, for any $\mathbf{u} \in \Z_p^n$. Similarly as the case $n$ is even, we can get
$\mathcal{H}_f(\mathbf{u}) = \pm \zeta_{p^k}^{f^{\ast}(\mathbf{u})}$.

(3) When $n$ is odd and $p\equiv 3 ~(mod~4)$, from Lemma \ref{Lemma2.3}, we have $\sqrt{p} \in \Q(\zeta_{4p}) \setminus \Q(\zeta_{2p})$.
It is easily to get that  $\mathcal{H}_f(\mathbf{u}) \in \Q(\zeta_{4p^k}) \setminus \Q(\zeta_{2p^k})$, from  Lemma \ref{Lemma2.4} and
condition (ii) of Lemma \ref{Lemma2.1}, we have $\mathcal{H}_f(\mathbf{u}) = \zeta_{4p^k}^{2l+1}$, for some $0 \leq l \leq 2p^k-1$.

If $p^k\equiv 1~(mod~4)$ and $l$ is even, then set $f^{\ast}(\mathbf{u})=\frac{3p^k+2l+1}{4} ~(mod~p^k)$, we have
$\mathcal{H}_f(\mathbf{u})=\sqrt{-1} \zeta_{p^k}^{f^{\ast}(\mathbf{u})}$; if $l$ is odd, then set
$f^{\ast}(\mathbf{u})=\frac{p^k+2l+1}{4} ~(mod~p^k)$, we have $\mathcal{H}_f(\mathbf{u})=-\sqrt{-1} \zeta_{p^k}^{f^{\ast}(\mathbf{u})}$.

If $p^k\equiv 3~(mod~4)$ and $l$ is even, then set $f^{\ast}(\mathbf{u})=\frac{p^k+2l+1}{4} ~(mod~p^k)$, we have
$\mathcal{H}_f(\mathbf{u})=-\sqrt{-1} \zeta_{p^k}^{f^{\ast}(\mathbf{u})}$; if $l$ is odd, then set
$f^{\ast}(\mathbf{u})=\frac{3p^k+2l+1}{4} ~(mod~p^k)$, we have $\mathcal{H}_f(\mathbf{u})=\sqrt{-1} \zeta_{p^k}^{f^{\ast}(\mathbf{u})}$.

This completes the proof.
\end{proof}

\begin{Remark}
\label{Remark2.1}
In Lemma \ref{Lemma2.5}, when $k=1$, the gbent function $f$ is just classical $p$-ary bent function. In this case, we have
\begin{eqnarray*}
\begin{split}
\mathcal{H}_f(\mathbf{u})=\left\{
\begin{array}{l}
\pm \zeta_p^{f^{\ast} (\mathbf{u})}  ~~~~~~~~~~~~$if$~~$n$~~is~~even~~or~~$n$~~is ~~odd~~and~~p \equiv 1(mod~4),\\
\pm \sqrt{-1} \zeta_p^{f^{\ast} (\mathbf{u})} ~~~~ $if$~~ $n$~~is ~~odd~~and~~p \equiv 3(mod~4).
\end{array}
\right.
\end{split}
\end{eqnarray*}
which  coincides with the result given by \c{C}e\c{s}melio\v{g}lu et al. in  \cite{Cesmelioglu}.
\end{Remark}

 \begin{Lemma}
 \label{Lemma2.6}
 Let $k$ is a positive integer, and $a\in \Z_p$,  then
 \begin{eqnarray*}
\begin{split}
\zeta_{p^k}^a=\frac{1}{p} \sum \limits_{i\in \Z_p} \left(\sum \limits_{j\in \Z_p} \zeta_p^{(a-i)j}\right) \zeta_{p^k}^i
\end{split}
\end{eqnarray*}
\end{Lemma}
\begin{proof}
Let $\mathcal{V}_p(\zeta_p)$ and $\mathcal{V}_p(\zeta_p^{-1})$ be the $p \times p$  matrix:
\begin{gather*}
\mathcal{V}_p(\zeta_p)=
\begin{pmatrix}
1 &  1 & \cdots & 1 \\
1 &  \zeta_p & \cdots & \zeta_p^{p-1}\\
\vdots & \vdots & \ddots & \vdots\\
1 &　\zeta_p^{p-1}  & \cdots & \zeta_p^{(p-1)(p-1)}
\end{pmatrix}
\end{gather*}
and
\begin{gather*}
\mathcal{V}_p(\zeta_p^{-1})=
\begin{pmatrix}
1 &  1 & \cdots & 1 \\
1 &  \zeta_p^{-1} & \cdots & \zeta_p^{-(p-1)}\\
\vdots & \vdots & \ddots & \vdots\\
1 &　\zeta_p^{-(p-1)}  & \cdots & \zeta_p^{-(p-1)(p-1)}
\end{pmatrix}.
\end{gather*}
In fact, we know that  $\mathcal{V}_p(\zeta_p)$ is a generalized Hadamard matrix, i.e., $\mathcal{V}_p(\zeta_p) (\mathcal{V}_p(\overline{\zeta_p}))^{\mathrm{T}}= p \mathrm{I}_p$,
and $(\mathcal{V}_p(\overline{\zeta_p}))^{\mathrm{T}}= \mathcal{V}_p(\zeta_p^{-1})$, therefore, we have
\begin{equation}
\label{03}
\mathcal{V}_p(\zeta_p)\mathcal{V}_p(\zeta_p^{-1}) = p \mathrm{I}_p,
\end{equation}
where $\mathrm{I}_p$ stands for the identity matrix of size $p$. Define now a collection of maps from $\C$ to itself by setting
\begin{equation*}
\left(
\begin{array}{c}
h_0(z)\\
h_1(z)\\
\vdots\\
h_{p-1}(z)
\end{array}
\right)  =
\mathcal{V}_p(\zeta_p^{-1})
\left(
\begin{array}{cccc}
1\\
z\\
\vdots\\
z^{p-1}
\end{array}
\right)
\end{equation*}
or equivalently, for any $i \in \Z_p$,
\begin{equation}
\label{04}
h_j(z)=\sum \limits_{i \in \Z_p} \zeta_p^{-ji}z^i.
\end{equation}
Furthermore, according to (\ref{03}), one has, for any $z \in \C$,\
\begin{equation*}
\left(
\begin{array}{cccc}
1\\
z\\
\vdots\\
z^{p-1}
\end{array}
\right)
 =\frac{1}{p}\mathcal{V}_p(\zeta_p)
\left(
\begin{array}{c}
h_0(z)\\
h_1(z)\\
\vdots\\
h_{p-1}(z)
\end{array}
\right)
\end{equation*}
that is, for any $a \in \Z_p$,
\begin{equation}
\label{05}
z^a=\frac{1}{p} \sum \limits_{j \in \Z_p} \zeta_p^{ja}h_j(z).
\end{equation}
Then plugging  Eq. (\ref{04}) into Eq. (\ref{05}),  we have
\begin{equation}
\label{06}
z^a=\frac{1}{p} \sum \limits_{i \in \Z_p} \Big(\sum \limits_{j \in \Z_p}  \zeta_p^{(a-i)j} \Big) z^i.
\end{equation}
If we set $z= \zeta_{q^k}$ and plugging it into Eq. (\ref{06}), then we get
\begin{equation*}
\zeta_{p^k}^a=\frac{1}{p} \sum \limits_{i \in \Z_p} \Big(\sum \limits_{j \in \Z_p}  \zeta_{p}^{(a-i)j} \Big) \zeta_{p^k}^i.
\end{equation*}
This completes the proof.
\end{proof}

\begin{Lemma}
\label{Lemma2.7}
$(i)$. $\{1, \zeta_{p^k},\zeta_{p^k}^2,\ldots,\zeta_{p^k}^{p^{k-1}-1}\}$ is a basis of  $\mathbb{Q}(\zeta_p, \zeta_{p^k})$
over $\mathbb{Q}(\zeta_p)$;\\
$(ii)$. $\{1, \zeta_{p^k},\zeta_{p^k}^2,\ldots,\zeta_{p^k}^{p^{k-1}-1}\}$ is a basis of  $\mathbb{Q}(\zeta_p,\sqrt{-1}, \zeta_{p^k})$
over $\mathbb{Q}(\zeta_p,\sqrt{-1})$.
\end{Lemma}
\begin{proof}
(i). It is easy to see that we only need to prove $\{1, \zeta_{p^k},\zeta_{p^k}^2,\ldots,\zeta_{p^k}^{p^{k-1}-1}\}$
is linear independently over $\mathbb{Q}(\zeta_p)$.\\
Suppose that there exists $a_i= \sum_{j=0}^{p-1}a_{ij}\zeta_p^{j} \in \mathbb{Q}(\zeta_p)$, $ 0\leq i \leq p^{k-1}-1$,
such that
\[\sum \limits_{i=0}^{p^{k-1}-1} a_i \zeta_{p^k}^{i}=0,\]
i.e.,
\[\sum \limits_{i=0}^{p^{k-1}-1} \sum \limits_{j=0}^{p-1}a_{ij}\zeta_p^{j} \zeta_{p^k}^{i}=
\sum \limits_{i=0}^{p^{k-1}-1} \sum \limits_{j=0}^{p-1}a_{ij} \zeta_{p^k}^{jp^{k-1}+i}=0.\]
It is well known that $\{1, \zeta_{p^k},\zeta_{p^k}^2,\ldots,\zeta_{p^k}^{(p-1)p^{k-1}-1}\}$ is a basis of $\mathbb{Q}(\zeta_{p^k})$ over $\mathbb{Q}$, thus
all $a_{ij}=0$, i.e., all $a_i=0$. So (i) holds.\\
(ii). Suppose that there exists $a_i= \sum_{j=0}^{p-1}a_{ij}\zeta_p^{j} \in \mathbb{Q}(\zeta_p,\sqrt{-1})$, $ 0\leq i \leq p^{k-1}-1$,
and $a_{ij}=b_{ij}+c_{ij}\sqrt{-1}$,
such that
\[\sum \limits_{i=0}^{p^{k-1}-1} a_i \zeta_{p^k}^{i}=0,\]
i.e.,
\begin{eqnarray*}
\begin{split}
\sum \limits_{i=0}^{p^{k-1}-1} \sum \limits_{j=0}^{p-1}a_{ij}\zeta_p^{j} \zeta_{p^k}^{i}&=
\sum \limits_{i=0}^{p^{k-1}-1} \sum \limits_{j=0}^{p-1}a_{ij} \zeta_{p^k}^{jp^{k-1}+i}\\
&=\sum \limits_{i=0}^{p^{k-1}-1} \sum \limits_{j=0}^{p-1}b_{ij} \zeta_{p^k}^{jp^{k-1}+i}
+\sqrt{-1}\sum \limits_{i=0}^{p^{k-1}-1} \sum \limits_{j=0}^{p-1}c_{ij} \zeta_{p^k}^{jp^{k-1}+i}\\
&=0.
\end{split}
\end{eqnarray*}
If $\sum_{i=0}^{p^{k-1}-1} \sum _{j=0}^{p-1}b_{ij} \zeta_{p^k}^{jp^{k-1}+i} \neq 0$, then $\sqrt{-1} \in \mathbb{Q}(\zeta_{p^k})$, which is a contradiction.
Therefore, $\sum_{i=0}^{p^{k-1}-1} \sum_{j=0}^{p-1}b_{ij} \zeta_{p^k}^{jp^{k-1}+i} =\sum_{i=0}^{p^{k-1}-1} \sum _{j=0}^{p-1}c_{ij} \zeta_{p^k}^{jp^{k-1}+i}= 0$,
similarly as (i), all $b_{ij}$ and $c_{ij}$ equal to 0, i.e., all $a_i=0$. So (ii) holds.
\end{proof}

\begin{Lemma}
\label{Lemma2.8}
Let $\gamma_{\textbf{a}}=\sum_{\textbf{v}\in \Z_{p}^{k-1}} \zeta_p^{-\textbf{a} \cdot \textbf{v}} \zeta_{p^k}^{\sum_{j=1}^{k-1}v_jp^{k-1-j}}$,
where $\textbf{a} \in \Z_p^{k-1}$ and $\textbf{v}=(v_1,v_2,\ldots, v_{k-1}) \in \Z_p^{k-1}$. Then
\[\zeta_{p^k}^{e}=\frac{1}{p^{k-1}} \sum \limits_{\textbf{a}\in \Z_p^{k-1}} \zeta_{p}^{\textbf{a}\cdot \textbf{u}} \gamma_{\textbf{a}} \]
where $e=\sum_{j=1}^{k-1} u_j p^{k-1-j}$ and $\textbf{u}=(u_1,u_2,\ldots, u_{k-1}) \in \Z_p^{k-1}$.
\end{Lemma}
\begin{proof}
For simplicity, denote
\[A=\frac{1}{p^{k-1}} \sum \limits_{\textbf{a}\in \Z_p^{k-1}} \zeta_{p}^{\textbf{a}\cdot \textbf{u}} \gamma_{\textbf{a}}.\]
Then, we have
\begin{eqnarray*}
\begin{split}
A&=&\frac{1}{p^{k-1}} \sum \limits_{\textbf{a}\in \Z_p^{k-1}} \zeta_{p}^{\textbf{a}\cdot \textbf{u}}
\sum_{\textbf{v}\in \Z_{p}^{k-1}} \zeta_p^{-\textbf{a} \cdot \textbf{v}} \zeta_{p^k}^{\sum_{j=1}^{k-1}v_jp^{k-1-j}}\\
&=&\frac{1}{p^{k-1}}\sum_{\textbf{v}\in \Z_{p}^{k-1}} \zeta_{p^k}^{\sum_{j=1}^{k-1}v_jp^{k-1-j}}
\sum \limits_{\textbf{a}\in \Z_p^{k-1}} \zeta_{p}^{\textbf{a}\cdot (\textbf{u}-\textbf{v})}.
\end{split}
\end{eqnarray*}
Since
$
\zeta_{p}^{\textbf{a}\cdot (\textbf{u}-\textbf{v})}=\left\{
\begin{array}{l}
0, ~~~~~~~~~~~~ \textbf{u}\neq \textbf{v}\\
p^{k-1},~~~~~~\textbf{u}= \textbf{v}
\end{array}
\right.
$.
This leads to
\[A=\frac{1}{p^{k-1}}  \cdot p^{k-1} \zeta_{p^k}^{\sum_{j=1}^{k-1}u_jp^{k-1-j}}= \zeta_{p^k}^e.\]
This completes the proof.
\end{proof}

\begin{Remark}
\label{Remark2.2}
Note that $\{1, \zeta_{p^k},\zeta_{p^k}^2,\ldots,\zeta_{p^k}^{p^{k-1}-1}\}$ is a basis of  $\mathbb{Q}(\zeta_p, \zeta_{p^k})$
over $\mathbb{Q}(\zeta_p)$, and $\zeta_{p^k}^e, 0 \leq e \leq p^{k-1}-1$, can be expressed by $\gamma_{\textbf{a}}, \textbf{a}\in \Z_p^{k-1}$, where the coefficients can form a non-singular matrix over  $\mathbb{Q}(\zeta_p)$. So $\{\gamma_{\textbf{a}} |　\textbf{a}\in \Z_p^{k-1}　\}$ is also a basis of $\mathbb{Q}(\zeta_p, \zeta_{p^k})$
over $\mathbb{Q}(\zeta_p)$. Similarly, $\{\gamma_{\textbf{a}} |　\textbf{a}\in \Z_p^{k-1}　\}$ is also a basis of $\mathbb{Q}(\zeta_p, \sqrt{-1}, \zeta_{p^k})$
over $\mathbb{Q}(\zeta_p,\sqrt{-1})$.
\end{Remark}

\begin{Lemma}
\label{Lemma2.9}
Let $\textbf{a} \in \Z_p^{k-1}$ and $\gamma_{\textbf{a}}=\sum_{\textbf{v}\in \Z_{p}^{k-1}} \zeta_p^{- \textbf{a}\cdot \textbf{v}} \zeta_{p^k}^{\sum_{i=1}^{k-1}v_ip^{k-1-i}}$.
Then
\[\gamma_{\textbf{a}}= \prod \limits_{i=1}^k\left( \sum \limits_{l \in \Z_{p}} \zeta_p^{la_i} \zeta_{p^{1+i}}^{p-l}\right)\]
\end{Lemma}
\begin{proof}
For simplicity, denote
\[A=\prod \limits_{i=1}^k\left( \sum \limits_{l \in \Z_{p}} \zeta_p^{la_i} \zeta_{p^{1+i}}^{p-l}\right).\]
Then, we have
\begin{eqnarray*}
\begin{split}
A&=\prod \limits_{i=1}^k\left( \sum \limits_{l \in \Z_{p}} \zeta_p^{la_i} \zeta_{p^{k}}^{(p-l)p^{k-1-i}}\right)\\
&\overset{v_i:=p-l}{=}\prod \limits_{i=1}^k\left( \sum \limits_{v_i \in \Z_{p}} \zeta_p^{-a_iv_i} \zeta_{p^{k}}^{v_ip^{k-1-i}}\right)\\
&=\sum \limits_{\textbf{v}\in \Z_p^{k-1}} \zeta_{p}^{- \textbf{a}\cdot \textbf{v}}\zeta_{p^k}^{\sum_{i=1}^{k-1}v_ip^{k-1-i}}.
\end{split}
\end{eqnarray*}
This completes the proof.
\end{proof}

\section{Sufficient (and necessary) conditions for gbent functions}
\label{sec3}
In this section, we mainly focus on functions from  $\Z_{p}^n$ to $\Z_q$, where $p$ is an odd prime number and $q$ is a positive integer
divided by $p$. In subsection \ref{sec3.1}, we present the complete characterization of gbent functions in $\mathcal{GB}_n^{p^k}$
in terms of classical $p$-ary bent functions. In other words, we give an efficient and necessary condition for gbent functions
in $\mathcal{GB}_n^{p^k}$. In subsection \ref{sec3.2}, we consider generalized bent functions in $\mathcal{GB}_n^{q}$, where
$q$ is a positive integer divided by $p$ but not the power of $p$. In this case, we give an sufficient condition for
weakly  regular gbent functions in $\mathcal{GB}_n^{q}$.

\subsection{sufficient and necessary conditions for gbent functions in $\mathcal{GB}_n^{p^k}$}
\label{sec3.1}

Let $f: \Z_p^n \rightarrow \Z_{p^k}$ be a generalized Boolean function  defined by $f(x)=\sum_{i=0}^{k-1}f_i(x)p^{k-1-i}$,
where  $f_i \in \mathcal{B}_n^p$. It turns out that the generalized Walsh-Hadamard spectrum
of $f$ can be described in terms of the Walsh-Hadamard spectrum of its components functions $f_i$.

\begin{thm}
\label{thm3.1.1}
Let $f(x)=\sum_{i=0}^{k-1}f_i(x)p^{k-1-i}$, where $f \in \mathcal{GB}_{n}^{p^k}$ and $f_i \in \mathcal{B}_{n}^p$.
Then
\begin{equation}
\label{07}
\mathcal{H}_f(\textbf{u})=\frac{1}{p^{k-1}} \sum \limits_{\textbf{a} \in \Z_p^{k-1}} \mathcal{H}_{f_0+\sum_{i=1}^{k-1}a_if_i}(\textbf{u}) \gamma_{\textbf{a}},
\end{equation}
where $\gamma_{\textbf{a}}=\sum_{\textbf{v}\in \Z_{p}^{k-1}} \zeta_p^{- \textbf{a}\cdot \textbf{v}} \zeta_{p^k}^{\sum_{i=1}^{k-1}v_ip^{k-1-i}}$.
\end{thm}
\begin{proof}
According to the definition of $\mathcal{H}_f({\textbf{u}})$, we have
\begin{eqnarray*}
p^{\frac{n}{2}}\mathcal{H}_f({\textbf{u}})&=&\sum \limits_{\textbf{x}\in \Z_{p}^{n}} \zeta_p^{- \textbf{u}\cdot \textbf{x}} \zeta_{p^k}^{f(x)}\\
&=&\sum \limits_{\textbf{x}\in \Z_{p}^{n}} \zeta_{p^k}^{\sum_{i=0}^{k-1}f_i(x)p^{k-1-i}} \zeta_p^{- \textbf{u}\cdot \textbf{x}}\\
&=&\sum \limits_{\textbf{x}\in \Z_{p}^{n}} \zeta_{p^{1+i}}^{\sum_{i=0}^{k-1}f_i(x)} \zeta_p^{- \textbf{u}\cdot \textbf{x}}\\
&=&\sum \limits_{\textbf{x}\in \Z_{p}^{n}} \zeta_p^{- \textbf{u}\cdot \textbf{x}} \prod \limits_{i=0}^{k-1} \zeta_{p^{1+i}}^{f_i(x)}\\
&=&\sum \limits_{\textbf{x}\in \Z_{p}^{n}} \zeta_p^{- \textbf{u}\cdot \textbf{x} + f_0(x)} \prod \limits_{i=1}^{k-1} \zeta_{p^{1+i}}^{f_i(x)}\\
&=&\frac{1}{p^{k-1}}\sum \limits_{\textbf{x}\in \Z_{p}^{n}} \zeta_p^{- \textbf{u}\cdot \textbf{x} + f_0(x)} \prod \limits_{i=1}^{k-1}
\left(\sum \limits_{j \in \Z_p} \Big(\sum \limits_{a_i \in \Z_P} \zeta_p^{(f_i(x)-j)a_i} \Big) \zeta_{p^{1+i}}^j\right) \\
&\overset{j:=p-l}{=}&\frac{1}{p^{k-1}}\sum \limits_{\textbf{x}\in \Z_{p}^{n}} \zeta_p^{- \textbf{u}\cdot \textbf{x} + f_0(x)} \prod \limits_{i=1}^{k-1}
\left(\sum \limits_{a_i \in \Z_p} \Big(\sum \limits_{l\in \Z_P} \zeta_p^{la_i} \zeta_{p^{1+i}}^{p-l} \Big) \zeta_{p}^{a_if_i(x)}\right) \\
&=&\frac{1}{p^{k-1}}\sum \limits_{\textbf{x}\in \Z_{p}^{n}} \zeta_p^{- \textbf{u}\cdot \textbf{x} + f_0(x)}
\left(\sum \limits_{\textbf{a} \in \Z_p} \zeta_{p}^{\sum_{i=1}^{k-1}a_if_i(x)} \prod \limits_{i=1}^{k-1}  \Big(\sum \limits_{l\in \Z_P} \zeta_p^{la_i} \zeta_{p^{1+i}}^{p-l} \Big) \right)\\
&=&\frac{1}{p^{k-1}}\sum \limits_{\textbf{a} \in \Z_p^{n}}  \prod \limits_{i=1}^{k-1}  \Big(\sum \limits_{l\in \Z_P} \zeta_p^{la_i} \zeta_{p^{1+i}}^{p-l} \Big)
\sum \limits_{\textbf{x}\in \Z_{p}^{k-1}} \zeta_p^{- \textbf{u}\cdot \textbf{x} + f_0(x)+\sum_{i=1}^{k-1}a_if_i(x)} \\
&=&\frac{1}{p^{k-1}}  p^{\frac{n}{2}}  \sum \limits_{\textbf{a} \in \Z_p^{n}}   \mathcal{H}_{ f_0(x)+\sum_{i=1}^{k-1}a_if_i(x)}(\textbf{u})
\prod \limits_{i=1}^{k-1}  \Big(\sum \limits_{l\in \Z_P} \zeta_p^{la_i} \zeta_{p^{1+i}}^{p-l} \Big)\\
&=&\frac{1}{p^{k-1}}  p^{\frac{n}{2}}  \sum \limits_{\textbf{a} \in \Z_p^{n}}   \mathcal{H}_{ f_0(x)+\sum_{i=1}^{k-1}a_if_i(x)}(\textbf{u})\gamma_{\textbf{a}}.
\end{eqnarray*}
The last equality holds from Lemma \ref{Lemma2.9}. It is easily to get that
\begin{equation*}
\mathcal{H}_f(\textbf{u})=\frac{1}{p^{k-1}} \sum \limits_{\textbf{a} \in \Z_p^{k-1}} \mathcal{H}_{f_0+\sum_{i=1}^{k-1}a_if_i}(\textbf{u}) \gamma_{\textbf{a}}.
\end{equation*}
This completes the proof.
\end{proof}

\begin{thm}
\label{thm3.1.2}
Let $k \geq 2$ and $f(x)=\sum_{i=0}^{k-1}f_i(x)p^{k-1-i}$, where $f \in \mathcal{GB}_{n}^{p^k}$ and $f_i \in \mathcal{B}_{n}^p$.
Then $f$ is gbent if and only if for any $\textbf{u} \in \Z_p^{n}$ and $\textbf{a}  \in \Z_p^{k-1}$, there exists some $\textbf{v} \in \Z_p^{k-1}$
and $j\in \Z_p$ such that
\begin{eqnarray*}
\begin{split}
\mathcal{H}_{f_0+\sum_{i=1}^{k-1}a_if_i}(\textbf{u})=\left\{
\begin{array}{l}
\pm \zeta_p^{\textbf{v}\cdot \textbf{a}+j}  ~~~~~~~~~~~~$if$~$n$~is~even~~or~~$n$~is ~odd~and~p \equiv 1(mod~4),\\
\pm  \sqrt{-1} \zeta_p^{\textbf{v}\cdot \textbf{a}+j} ~~~~ $if$~ $n$~is ~odd~and~p \equiv 3(mod~4).
\end{array}
\right.
\end{split}
\end{eqnarray*}
where $\mathbf{a}=(a_1,a_2,\cdots,a_{k-1}).$
\end{thm}
\begin{proof}
If $n$ is even or $n$ is odd and $p\equiv 1~(mod~4)$, then $\mathcal{H}_f{(\textbf{u})}= \pm \zeta_{p^k}^i$, for some $0 \leq i \leq p^k-1$.
Hence, $\mathcal{H}_f{(\textbf{u})}$ can be expressed as $\mathcal{H}_f{(\textbf{u})}= \pm \zeta_p^j \zeta_{p^k}^{i-jp^{k-1}}$,
where $0 \leq j \leq p-1$ and $0 \leq i-jp^{k-1}  \leq p^{k-1}-1$.
According to Theorem \ref{thm3.1.1} and Lemma \ref{Lemma2.8}, we have
\begin{eqnarray*}
\mathcal{H}_f{(\textbf{u})}&=&\frac{1}{p^{k-1}} \sum \limits_{\textbf{a} \in \Z_p^{k-1}} \mathcal{H}_{f_0+\sum_{i=1}^{k-1}a_if_i}(\textbf{u}) \gamma_{\textbf{a}}\\
&=&\pm \zeta_p^j \frac{1}{p^{k-1}} \sum \limits_{\textbf{a}\in \Z_p^{k-1}} \zeta_{p}^{\textbf{a}\cdot \textbf{v}} \gamma_{\textbf{a}}
\end{eqnarray*}
By the definition of Walsh-Hadamard transform and condition (i) of Lemma \ref{Lemma2.3}, we have
\[\mathcal{H}_{f_0+\sum_{i=1}^{k-1}a_if_i}(\textbf{u}) = p^{-\frac{n}{2}}\sum_{\mathbf{x} \in \Z_p^n} \zeta_p^{f_0(\mathbf{x})+\sum_{i=1}^{k-1}a_if_i(\mathbf{x})}
\in \Q(\zeta_p),\]
when $n$ is even or $n$ is odd and $p\equiv 1~(mod~4)$.
Then from Remark \ref{Remark2.2}, we can get
\[\mathcal{H}_{f_0+\sum_{i=1}^{k-1}a_if_i}(\textbf{u})=\pm \zeta_p^{\textbf{v}\cdot \textbf{a}+j},\]
where $\mathbf{v}$ and $j$ only depend $f$ and $\mathbf{u}$.

If $n$ is odd and $p\equiv 3~(mod~4)$, By the definition of Walsh-Hadamard transform and condition (ii) of Lemma \ref{Lemma2.3}, we have
\[\mathcal{H}_{f_0+\sum_{i=1}^{k-1}a_if_i}(\textbf{u}) = p^{-\frac{n}{2}}\sum_{\mathbf{x} \in \Z_p^n} \zeta_p^{f_0(\mathbf{x})+\sum_{i=1}^{k-1}a_if_i(\mathbf{x})}
\in \Q(\zeta_{4p}) \setminus \Q(\zeta_{2p})  \subseteq \Q(\zeta_{p},\sqrt{-1}) .\]
Similarly as above, from Lemma \ref{Lemma2.5} and Remark  \ref{Remark2.2}, we can get
\[\mathcal{H}_{f_0+\sum_{i=1}^{k-1}a_if_i}(\textbf{u})=\pm \sqrt{-1} \zeta_p^{\textbf{v}\cdot \textbf{a}+j}.\]
This completes the proof.
\end{proof}

\begin{Remark}
\label{Remark3.1.1}
Let $k\geq 2$ and $f(x)=\sum_{i=0}^{k-1}f_i(x)p^{k-1-i}$ be a gbent function, where $f \in \mathcal{GB}_{n}^{p^k}$ and $f_i \in \mathcal{B}_{n}^p$.
Then, from Theorem \ref{thm3.1.2}, $f_0+\sum_{i=1}^{k-1}a_if_i$ is $p$-ary bent for all $\mathbf{a} \in \Z_{p}^{k-1}$.
\end{Remark}

The following results are straightforward consequences of Theorem \ref{thm3.1.2}.

\begin{Corollary}
\label{Corollary3.1.1}
Let $k\geq 2$ and $f(x)=\sum_{i=0}^{k-1}f_i(x)p^{k-1-i}$ be a gbent function, where $f \in \mathcal{GB}_{n}^{p^k}$ and $f_i \in \mathcal{B}_{n}^p$.
Then, $g_{\pi}$ is always gbent, where $g_{\pi}$ is defined as $g_{\pi}(\mathbf{x})=g_0(\mathbf{x})p^{k-1}  + \sum_{i=1}^{k-1}g_{\pi(i)}(\mathbf{x})p^{k-1-i} $
for any permutation $\pi$ of $\{1,2,\cdots,k-1\}$.
\end{Corollary}

\begin{Corollary}
\label{Corollary3.1.2}
Let $k\geq 2$, $l \leq k$, and $f(x)=\sum_{i=0}^{k-1}f_i(x)p^{k-1-i}$ be a gbent function, where $f \in \mathcal{GB}_{n}^{p^k}$ and $f_i \in \mathcal{B}_{n}^p$.
Then, $g_{I}$ is always gbent in $\mathcal{GB}_n^{p^l}$, where $g_{I}$ is defined as $g_{I}(\mathbf{x})=g_0(\mathbf{x})p^{l-1}  + \sum_{j=1}^{l-1}g_{i_j}(\mathbf{x})p^{k-1-i} $
for any subset $I=\{i_1,\cdots,i_{l-1}\}$ of $\{1,2,\cdots,k-1\}$, where $\#I=l-1$.
\end{Corollary}

\subsection{sufficient conditions for regular gbent functions in $\mathcal{GB}_n^{q}$}
\label{sec3.2}
In this subsection, we consider generalized Boolean functions in $\mathcal{GB}_n^{q}$, where
$q$ is a positive integer divided by $p$ but not the power of $p$. In this case,  an efficient condition for
weakly regular gbent function was given, which in terms of the component functions of $f$.

Let $f:\Z_p^n \rightarrow \Z_q$, $p^{k-1} < q < p^{k}$  and  $p|q$,  be given, commonly, as $f(x)=\sum_{i=0}^{k-1}f_i(x)p^{k-1-i}$.
However, for the reasons explained below, we rewrite the function $f(x)$ as
\begin{equation}
\label{08}
f(x)=\frac{q}{p}f_0(x)+\sum \limits_{i=1}^{k-1}f_i(x)p^{k-1-i},
\end{equation}
where $f_i \in \mathcal{B}_n^p$. The importance of the term $\frac{q}{p}f_0(x)$ is due to the fact that $\frac{q}{p}$ is
the only coefficient from $\Z_q$ for which it holds that $\zeta_q^{\frac{q}{p}f_0(x)} = \zeta_p^{f_0(x)}$. This coefficients, which naturally
appears when $q=p^k$  as the coefficient of $f_0(x)$ as discussed in subsection \ref{sec3.1}, actually made it possible to express the spectral values
of the generalized Walsh-Hadamard transform of $f$ in terms of certain linear combinations of $\mathcal{H}_{f_0+\sum_{i=1}^{k-1}a_if_i}(\textbf{u})$
as given by (\ref{07}). Then, we give the main result of this subsection.

\begin{thm}
\label{thm3.2.1}
Let $f:\Z_p^n \rightarrow \Z_q$, $p^{k-1} < q < p^{k}$ and  $p|q$, be given as $f(x)=\frac{q}{p}f_0(x)+\sum_{i=1}^{k-1}f_i(x)p^{k-1-i}$,  where $f_i \in \mathcal{B}_n^p$.
For any $\textbf{u} \in \Z_p^{n}$ and $ \textbf{a} \in \Z_p^{k-1}$, if there exists some $\textbf{v} \in \Z_p^{k-1}$ and $j\in \Z_p$ such that
\begin{eqnarray*}
\mathcal{H}_{f_0+\sum_{i=1}^{k-1}a_if_i}(\textbf{u})= \alpha \zeta_p^{\textbf{v}\cdot \textbf{a}+j},
\end{eqnarray*}
where $\mathbf{a}=(a_1,a_2,\cdots,a_{k-1})$ and $\alpha$ is a complex number with unit magnitude, then $f$ is a weakly regular gbent function.
\end{thm}
\begin{proof}
Similarly as in Theorem \ref{thm3.1.1}, we can get
\begin{equation*}
\mathcal{H}_f(\textbf{u})=\frac{1}{p^{k-1}} \sum \limits_{\textbf{a} \in \Z_p^{k-1}} \mathcal{H}_{f_0+\sum_{i=1}^{k-1}a_if_i}(\textbf{u}) \widetilde{\gamma_{\textbf{a}}}.
\end{equation*}
Note that $\widetilde{\gamma_{\textbf{a}}}=\sum_{\textbf{v}\in \Z_{p}^{k-1}} \zeta_p^{- \textbf{a}\cdot \textbf{v}} \zeta_{q}^{\sum_{i=1}^{k-1}v_jp^{k-1-i}}$.
For any $\textbf{u}, \textbf{a} \in \Z_p^{k-1}$, if there exists some $\textbf{v} \in \Z_p^{k-1}$ and $j\in \Z_p$ such that
\begin{eqnarray*}
\mathcal{H}_{f_0+\sum_{i=1}^{k-1}a_if_i}(\textbf{u})=\alpha \zeta_p^{\textbf{v}\cdot \textbf{a}+j},
\end{eqnarray*}
then we have
\begin{eqnarray*}
\mathcal{H}_f(\textbf{u})&=& \frac{1}{p^{k-1}} \sum \limits_{\textbf{a} \in \Z_p^{k-1}} \alpha \zeta_p^{\textbf{v}\cdot \textbf{a}+j}   \widetilde{\gamma_{\textbf{a}}}\\
&=&\frac{\alpha}{p^{k-1}} \sum \limits_{\textbf{a} \in \Z_p^{k-1}} \zeta_p^{\textbf{v}\cdot \textbf{a}+j}
\sum \limits_{\textbf{w}\in \Z_{p}^{k-1}} \zeta_p^{- \textbf{a}\cdot \textbf{w}} \zeta_{q}^{\sum_{i=1}^{k-1}w_ip^{k-1-i}}\\
&=& \frac{\alpha}{p^{k-1}}  \zeta_p^{j} \sum \limits_{\textbf{w}\in \Z_{p}^{k-1}}  \zeta_{q}^{\sum_{i=1}^{k-1}w_ip^{k-1-i}}
\sum \limits_{\textbf{a} \in \Z_p^{k-1}} \zeta_p^{(\textbf{v}-\textbf{w})\cdot \textbf{a}} \\
&=&\alpha \zeta_p^{j}  \zeta_{q}^{\sum_{i=1}^{k-1}v_ip^{k-1-i}}\\
&=&\alpha  \zeta_{q}^{\frac{qj}{p}  + \sum_{i=1}^{k-1}v_ip^{k-1-i}}.
\end{eqnarray*}
Therefore, we conclude that $f$ is a weakly regular gbent function. This completes the proof.
\end{proof}

\begin{Remark}
\label{Remark3.2.1}
Under the assumptions of Theorem  {\ref{thm3.2.1}}, it is easily to get that $f_{\pi}(\mathbf{x})$ is always  weakly regular
gbent, where $f_{\pi}(\mathbf{x})$ is defined as $f_{\pi}(\mathbf{x}) = \frac{q}{p}f_0(x)+\sum_{i=1}^{k-1}f_{\pi}(x)p^{k-1-i}$
for any permutation $\pi$ of $\{1,2,\cdots, k-1\}$.
\end{Remark}

\begin{Remark}
\label{Remark3.2.2}
For gbent functions in $\mathcal{GB}_n^{p^k}$, we have proved that $\mathcal{H}_f(\textbf{u}) = \alpha \zeta_{p^k}^{f^{\ast}(\mathbf{u})}$
for any $\mathbf{u} \in \Z_p^n$  and $\alpha \in \{\pm 1, \pm \sqrt{-1} \}$ in section \ref{sec2}. However, for gbent functions in $\mathcal{GB}_n^{q}$ $(p^{k-1}<q<p^k)$,
where $q$ is divided by $p$ but not the power of $p$, we can not completely determine what $\mathcal{H}_f(\textbf{u})$ can attain for
any $\mathbf{u} \in \Z_p^n$.

If we assume that $\alpha \in \{\pm 1, \pm \sqrt{-1} \}$ in theorem \ref{thm3.2.1}, we still can not get sufficient and necessary conditions for
gbent functions in $\mathcal{GB}_n^{q}$, because, in this case,  $\{ \widetilde{\gamma_{\mathbf{a}}} |~   \mathbf{a} \in \Z_p^{k-1}\}$ is not a
basis of  $\Q (\zeta_q)$ over $\Q (\zeta_p)$  (since $\frac{\phi(q)}{p-1}< p^{k-1}$, where~$\phi$~is~Euler's  totient~function).
\end{Remark}

\section{Construction of gbent functions}
\label{sec4}
The results from previous sections allow us to construct gbent functions in $\mathcal{GB}_n^{p^k}$ (or $\mathcal{GB}_n^{q}$).
In fact, the conditions in Theorem \ref{thm3.1.2}  can be transformed in the following words.
For $(a_1,a_2,\cdots,a_{k-1}) \in \Z_p^n$, let $i:=\sum_{j=1}^{k-1}a_jp^{j-1}$ and denote $\mathbf{a}_i=(a_1,a_2,\cdots,a_{k-1})$.
Denote $\mathcal{H}_{\textbf{a}_i}(\textbf{u})=\mathcal{H}_{f_0+\sum_{i=1}^{k-1}a_if_i}(\textbf{u}) = \alpha \zeta_p^j \zeta_p^{\textbf{v}\cdot \textbf{a}}$,
where $\alpha \in \{\pm 1, \pm \sqrt{-1} \}$. It is easily to see that  for any $\mathbf{u} \in \Z_p^{n}$, we have
 $( \mathcal{H}_{\textbf{a}_0}(\textbf{u}), \mathcal{H}_{\textbf{a}_1}(\textbf{u}), \cdots, \mathcal{H}_{\textbf{a}_{p^{k-1}-1}}(\textbf{u}) ) =
 \alpha \zeta_p^j (\zeta_p^{\textbf{v}\cdot \textbf{a}_0},  \zeta_p^{\textbf{v}\cdot \textbf{a}_1}, \cdots,
\zeta_p^{\textbf{v}\cdot \textbf{a}_{p^{k-1}-1}}  )  = \alpha \zeta_p^j  H_{p^{k-1}}^{(r)} $, where $H_{p^{k-1}} = H_p^{\otimes{(k-1)}}$   is a generalized
Hadamad matrix of order $p^{k-1}$ and $H_p$ is defined as $H_p=(\zeta_p^{i\cdot j})_{0 \leq i,j \leq p-1}$,  $H_{p^{k-1}}^{(r)}$ stands for the $r$-th row of $H_{p^{k-1}}$.
Therefore, Theorem \ref{thm3.1.2} can restate as follows:

Let $k \geq 2$ and $f(x)=\sum_{i=0}^{k-1}f_i(x)p^{k-1-i} \in \mathcal{GB}_{n}^{p^k}$ and $f_i \in \mathcal{B}_{n}^p$.
Then $f$ is gbent if and only if for every $\textbf{u} \in \Z_p^{n}$, there exists some $\textbf{v} \in \Z_p^{k-1}$
and $j\in \Z_p$ such that
\begin{eqnarray*}
\begin{split}
( \mathcal{H}_{\textbf{a}_0}(\textbf{u}), \mathcal{H}_{\textbf{a}_1}(\textbf{u}), \cdots, \mathcal{H}_{\textbf{a}_{p^{k-1}-1}}(\textbf{u}) )=\left\{
\begin{array}{l}
\pm \zeta_p^{j} (\zeta_p^{\textbf{v}\cdot \textbf{a}_0},  \zeta_p^{\textbf{v}\cdot \textbf{a}_1}, \cdots,
\zeta_p^{\textbf{v}\cdot \textbf{a}_{p^{k-1}-1}}  ) \\
~$if$~$n$~is~even~~or~~$n$~is ~odd~and~p \equiv 1(mod~4),\\
\\
\pm  \sqrt{-1} \zeta_p^j (\zeta_p^{\textbf{v}\cdot \textbf{a}_0},  \zeta_p^{\textbf{v}\cdot \textbf{a}_1}, \cdots,
\zeta_p^{\textbf{v}\cdot \textbf{a}_{p^{k-1}-1}}  )\\
 ~ $if$~ $n$~is ~odd~and~p \equiv 3(mod~4).
\end{array}
\right.
\end{split}
\end{eqnarray*}
where $\mathbf{a}=(a_1,a_2,\cdots,a_{k-1}).$

\begin{thm}
\label{thm4.1}
Let $k \geq 2$, $n=2m$,  and $f(x)=\sum_{i=0}^{k-1}f_i(x)p^{k-1-i} \in \mathcal{GB}_{2m}^{p^k}$. Let $f_0(x)=\sum_{i=1}^m \beta_ix_ix_{i+m}$  and
$f_i(x) \in \mathcal{A}(x_1,x_2,\cdots,x_m;\Z_p)$, $1 \leq i \leq m$, where $\beta_i \in \Z_p$  and $\mathcal{A}(x_1,x_2,\cdots,x_m;\Z_p)$ denotes
the set of all affine functions from $\Z_p^n \rightarrow \Z_p$, with variables $x_1,\cdots, x_m$ and coefficients in $\Z_p$.
Then $f$ is gbent.
\end{thm}

\begin{proof}
With the help of the previous discussion, we only need to prove that for every $\mathbf{a}_i=(a_1, \cdots, a_{k-1}) \in \Z_p^{k-1}$ and $\textbf{u} \in \Z_p^{n}$
there exists some $\textbf{v} \in \Z_p^{k-1}$ and $j\in \Z_p$ such that
\[\mathcal{H}_{\textbf{a}_i}(\textbf{u}) = \epsilon \zeta_p^j \cdot \zeta_p^{\textbf{v}\cdot \textbf{a}_i},\]
where $\epsilon= 1$ or $-1$ is independent of $\textbf{a}_i$ (in this case $\epsilon=1$).\\
By definition
\begin{eqnarray*}
\mathcal{H}_{\textbf{a}_i}(\textbf{u}) &=& p^{-m} \sum \limits_{ \mathbf{x} \in \Z_p^{2m}} \zeta_p^{f_0(\mathbf{x})+\sum_{i=1}^{k-1}a_if_i(\mathbf{x}) - \mathbf{u}\cdot\mathbf{x}}\\
 &=&p^{-m} \sum \limits_{ \mathbf{x} \in \Z_p^{2m}} \zeta_p^{\sum_{i=1}^m\beta_ix_ix_{i+m} + \sum_{i=1}^{k-1}a_il_i(x_1,\cdots,x_m)- \sum_{i=1}^{2m}u_ix_i }\\
 &=&p^{-m} \left(\sum \limits_{ (x_1,x_{m+1}) \in \Z_p^{2}}  \cdots   \left(\sum \limits_{(x_{m},x_{2m}) \in \Z_p^2}  \zeta_p^{\beta_mx_mx_{2m} +  x_{m} l^{(m)}(a_1,\cdots, a_{k-1})- \sum_{i=1}^2 u_{mi}x_{mi} } \right) \right)\\
 &=&p^{-m} \left(\sum \limits_{ (x_1,x_{m+1}) \in \Z_p^{2}}  \cdots   \left(\sum \limits_{x_{m} \in \Z_p} \zeta_p^{x_{m} l^{(m)}(a_1,\cdots, a_{k-1})-u_mx_m } \sum \limits_{x_{2m} \in \Z_p}  \zeta_p^{\beta_mx_mx_{2m} - u_{2m}x_{2m} } \right)  \right)\\
 &=&p^{-m} \left(\sum \limits_{ (x_1,x_{m+1}) \in \Z_p^{2}}  \cdots   \left( p\zeta_p^{ u_{2m} \beta_m^{-1}(l^{(m)}(a_1,\cdots, a_{k-1})-u_m)}   \right)  \right)\\
 &=&  \zeta_p^{ \sum_{i=1}^m u_{2i} \beta_i^{-1}(l^{(i)}(a_1,\cdots, a_{k-1})-u_i)} \\
 &=&  \zeta_p^{g_0(\beta_1,\cdots,\beta_m;u_1,\cdots, u_{2m}) + \sum_{i=1}^{k-1}  a_ig_i(\beta_1,\cdots,\beta_m;u_1,\cdots, u_{2m})} \\
 &=&  \zeta_p^j \zeta_p^{\textbf{v}\cdot \textbf{a}_i}.
\end{eqnarray*}
Where $j=g_0(\beta_1,\cdots,\beta_m;u_1,\cdots, u_{2m})$ and $\mathbf{v}  = \big(g_1(\beta_1,\cdots,\beta_m;u_1,\cdots, u_{2m}), \\ \cdots, g_{k-1}(\beta_1,\cdots,\beta_m;u_1,\cdots, u_{2m})\big)$,
$g_i$ is a  function with variables $\beta_1,\cdots,\beta_m$,  $u_1,\cdots, u_{2m}$,  and $l^{(i)}(a_1,\cdots, a_{k-1}) \in \mathcal{A}(a_1,\cdots, a_{k-1};\Z_p)$.

This completes the proof.
\end{proof}

\begin{Remark}
\label{4.1}
we note that the above discussions and  Theorem \ref{thm4.1} is also adapt to gbent functions in $\mathcal{GB}_n^q$, where $q$ is divided by $p$ and
not the power of it.
\end{Remark}

\begin{Corollary}
\label{Corollary4.1}
Let $k \geq 2$, $n=2m$,  and $f(x)=\sum_{i=0}^{k-1}f_i(x)p^{k-1-i} \in \mathcal{GB}_{2m}^{p^k}$. Let $f_0(x)=\sum_{i=1}^m \beta_ix_ix_{i+m}$  and
$f_i(x) =c_i$, where $\beta_i, c_i \in \Z_p$,  $1 \leq i \leq m$, then $f$ is gbent.
\end{Corollary}
\begin{proof}
From Theorem \ref{thm4.1}, the corollary follows.
\end{proof}

In the following, we give the examples of gbent functions for $q=p^k$ and not, respectively.
\begin{Example}
Let $f: \Z_3^4 \rightarrow \Z_{27}$ and $f(x)=9f_0(x)+3f_1(x)+f_2(x)$, where $f_0(x)=2x_1x_3+x_2x_4$, $f_1(x)=x_1+x_2$ and $f_2(x)=x_1$.
From Theorem \ref{thm4.1}, we know that $f$ is gbent.

Denoting $\mathcal{H}_{\textbf{a}}(\textbf{u}) =  ( \mathcal{H}_{\textbf{a}_0}(\textbf{u}), \mathcal{H}_{\textbf{a}_1}(\textbf{u}), \cdots, \mathcal{H}_{\textbf{a}_{8}}(\textbf{u}) )$,
for $\mathbf{u} \in \Z_3^4$, the vectors $\mathcal{H}_{\textbf{a}}(\textbf{u})$ are given in Table \ref{table1}. We note that $H_9=H_3^{\otimes 2}$, where
\begin{gather*}
H_3=
\begin{pmatrix}
1 &  1 &  1 \\
1 &  \zeta_3 & \zeta_3^2\\
1 &　\zeta_3^2  & \zeta_3
\end{pmatrix}.
\end{gather*}
\end{Example}

\begin{Example}
Let $f: \Z_3^4 \rightarrow \Z_{21}$ and $f(x)=7f_0(x)+3f_1(x)+f_2(x)$, where $f_0(x)=x_1x_3+2x_2x_4$, $f_1(x)=2x_1+x_2$ and $f_2(x)=1$.
From Theorem \ref{thm4.1} and \ref{thm3.2.1}, we know that $f$ is gbent.

Denoting $\mathcal{H}_{\textbf{a}}(\textbf{u}) =  ( \mathcal{H}_{\textbf{a}_0}(\textbf{u}), \mathcal{H}_{\textbf{a}_1}(\textbf{u}), \cdots, \mathcal{H}_{\textbf{a}_{8}}(\textbf{u}) )$,
for $\mathbf{u} \in \Z_3^4$, the vectors $\mathcal{H}_{\textbf{a}}(\textbf{u})$ are given in Table \ref{table2}.
\end{Example}

\section{Concluding remarks}
\label{sec5}
In this paper, we have investigated gbent functions in $\mathcal{GB}_{n}^{q}$.  For $q=p^k$, the complete characterization of gbent functions in $\mathcal{GB}_n^{p^k}$
in terms of classical $p$-ary bent functions are given, and for  such  $q$   divided by $p$ but not a power of it, we give the sufficient conditions for weakly
regular gbent functions. Besides, constructions of  such gbent functions in $\mathcal{GB}_n^{q}$ is also considered.  It would be
interesting to construct more gbent  functions  from $\Z_p^{n}$ to $ \Z_q$.



\begin{appendices}
\section{ $\mathcal{H}_{\textbf{a}}(\textbf{u})$ for $f$ in Example 4.1 }
\begin{longtable}{|c|c|c|}
\caption{Vectors $\mathcal{H}_{\textbf{a}}(\textbf{u})$ for all $\mathbf{u} \in \Z_3^4$} \label{table1}\\
\hline
$\mathbf{u} \in \Z_3^4$  &　　$\mathcal{H}_{\textbf{a}}(\textbf{u}) =  ( \mathcal{H}_{\textbf{a}_0}(\textbf{u}), \cdots, \mathcal{H}_{\textbf{a}_{8}}(\textbf{u}) )$　　&　$\mathcal{H}_{\textbf{a}}(\textbf{u}) = \zeta_3^iH_{9}^{(r)}$\\
\hline

\hline
$\mathbf{u_0} =(0,0,0,0)$     &  (1,1,1,1,1,1,1,1,1)        &      $H_{9}^{(0)}$\\
\hline
$\mathbf{u_1} =(1,0,0,0)$     &  (1,1,1,1,1,1,1,1,1)        &      $H_{9}^{(0)}$\\
\hline
$\mathbf{u_2} =(2,0,0,0)$     &  (1,1,1,1,1,1,1,1,1)        &      $H_{9}^{(0)}$\\
\hline
$\mathbf{u_3} =(0,1,0,0)$     &  (1,1,1,1,1,1,1,1,1)        &      $H_{9}^{(0)}$\\
\hline
$\mathbf{u_4} =(1,1,0,0)$     &  (1,1,1,1,1,1,1,1,1)        &      $H_{9}^{(0)}$\\
\hline
$\mathbf{u_5} =(2,1,0,0)$     &  (1,1,1,1,1,1,1,1,1)        &      $H_{9}^{(0)}$\\
\hline
$\mathbf{u_6} =(0,2,0,0)$     &  (1,1,1,1,1,1,1,1,1)        &      $H_{9}^{(0)}$\\
\hline
$\mathbf{u_7} =(1,2,0,0)$     &  (1,1,1,1,1,1,1,1,1)        &      $H_{9}^{(0)}$\\
\hline
$\mathbf{u_8} =(2,2,0,0)$     &  (1,1,1,1,1,1,1,1,1)        &      $H_{9}^{(0)}$\\
\hline
$\mathbf{u_9} =(0,0,1,0)$     &  $(1,\zeta_3^2,\zeta_3,\zeta_3^2,\zeta_3,1,\zeta_3,1,\zeta_3^2)$        &      $H_{9}^{(8)}$\\
\hline
$\mathbf{u_{10}} =(1,0,1,0)$     &  $\zeta_3(1,\zeta_3^2,\zeta_3,\zeta_3^2,\zeta_3,1,\zeta_3,1,\zeta_3^2)$        &      $\zeta_3H_{9}^{(8)}$\\
\hline
$\mathbf{u_{11}} =(2,0,1,0)$     &  $\zeta_3^2(1,\zeta_3^2,\zeta_3,\zeta_3^2,\zeta_3,1,\zeta_3,1,\zeta_3^2)$        &      $\zeta_3^2H_{9}^{(8)}$\\
\hline
$\mathbf{u_{12}} =(0,1,1,0)$     &  $(1,\zeta_3^2,\zeta_3,\zeta_3^2,\zeta_3,1,\zeta_3,1,\zeta_3^2)$        &      $H_{9}^{(8)}$\\
\hline
$\mathbf{u_{13}} =(1,1,1,0)$     &  $\zeta_3(1,\zeta_3^2,\zeta_3,\zeta_3^2,\zeta_3,1,\zeta_3,1,\zeta_3^2)$        &      $\zeta_3H_{9}^{(8)}$\\
\hline
$\mathbf{u_{14}} =(2,1,1,0)$     &  $\zeta_3^2(1,\zeta_3^2,\zeta_3,\zeta_3^2,\zeta_3,1,\zeta_3,1,\zeta_3^2)$        &      $\zeta_3^2H_{9}^{(8)}$\\
\hline
$\mathbf{u_{15}} =(0,2,1,0)$     &  $(1,\zeta_3^2,\zeta_3,\zeta_3^2,\zeta_3,1,\zeta_3,1,\zeta_3^2)$        &      $H_{9}^{(8)}$\\
\hline
$\mathbf{u_{16}} =(1,2,1,0)$     &  $\zeta_3(1,\zeta_3^2,\zeta_3,\zeta_3^2,\zeta_3,1,\zeta_3,1,\zeta_3^2)$        &      $\zeta_3H_{9}^{(8)}$\\
\hline
$\mathbf{u_{17}} =(2,2,1,0)$     &  $\zeta_3^2(1,\zeta_3^2,\zeta_3,\zeta_3^2,\zeta_3,1,\zeta_3,1,\zeta_3^2)$        &      $\zeta_3^2H_{9}^{(8)}$\\
\hline
$\mathbf{u_{18}} =(0,0,2,0)$     &  $(1,\zeta_3,\zeta_3^2,\zeta_3,\zeta_3^2,1,\zeta_3^2,1,\zeta_3)$        &      $H_{9}^{(4)}$\\
\hline
$\mathbf{u_{19}} =(1,0,2,0)$     &  $\zeta_3^2(1,\zeta_3,\zeta_3^2,\zeta_3,\zeta_3^2,1,\zeta_3^2,1,\zeta_3)$        &      $\zeta_3^2H_{9}^{(4)}$\\
\hline
$\mathbf{u_{20}} =(2,0,2,0)$     &  $\zeta_3(1,\zeta_3,\zeta_3^2,\zeta_3,\zeta_3^2,1,\zeta_3^2,1,\zeta_3)$        &      $\zeta_3H_{9}^{(4)}$\\
\hline
$\mathbf{u_{21}} =(0,1,2,0)$     &  $(1,\zeta_3,\zeta_3^2,\zeta_3,\zeta_3^2,1,\zeta_3^2,1,\zeta_3)$        &      $H_{9}^{(4)}$\\
\hline
$\mathbf{u_{22}} =(1,1,2,0)$     &  $\zeta_3^2(1,\zeta_3,\zeta_3^2,\zeta_3,\zeta_3^2,1,\zeta_3^2,1,\zeta_3)$        &      $\zeta_3^2H_{9}^{(4)}$\\
\hline
$\mathbf{u_{23}} =(2,1,2,0)$     &  $\zeta_3(1,\zeta_3,\zeta_3^2,\zeta_3,\zeta_3^2,1,\zeta_3^2,1,\zeta_3)$        &      $\zeta_3H_{9}^{(4)}$\\
\hline
$\mathbf{u_{24}} =(0,2,2,0)$     &  $(1,\zeta_3,\zeta_3^2,\zeta_3,\zeta_3^2,1,\zeta_3^2,1,\zeta_3)$        &      $H_{9}^{(4)}$\\
\hline
$\mathbf{u_{25}} =(1,2,2,0)$     &  $\zeta_3^2(1,\zeta_3,\zeta_3^2,\zeta_3,\zeta_3^2,1,\zeta_3^2,1,\zeta_3)$        &      $\zeta_3^2H_{9}^{(4)}$\\
\hline
$\mathbf{u_{26}} =(2,2,2,0)$     &  $\zeta_3(1,\zeta_3,\zeta_3^2,\zeta_3,\zeta_3^2,1,\zeta_3^2,1,\zeta_3)$        &      $\zeta_3H_{9}^{(4)}$\\
\hline
$\mathbf{u_{27}} =(0,0,0,1)$     &  $(1,1,1,\zeta_3,\zeta_3,\zeta_3,\zeta_3^2,\zeta_3^2,\zeta_3^2)$        &      $H_{9}^{(3)}$\\
\hline
$\mathbf{u_{28}} =(1,0,0,1)$     &  $(1,1,1,\zeta_3,\zeta_3,\zeta_3,\zeta_3^2,\zeta_3^2,\zeta_3^2)$        &      $H_{9}^{(3)}$\\
\hline
$\mathbf{u_{29}} =(2,0,0,1)$     &  $(1,1,1,\zeta_3,\zeta_3,\zeta_3,\zeta_3^2,\zeta_3^2,\zeta_3^2)$        &      $H_{9}^{(3)}$\\
\hline
$\mathbf{u_{30}} =(0,1,0,1)$     &  $\zeta_3^2(1,1,1,\zeta_3,\zeta_3,\zeta_3,\zeta_3^2,\zeta_3^2,\zeta_3^2)$        &      $\zeta_3^2H_{9}^{(3)}$\\
\hline
$\mathbf{u_{31}} =(1,1,0,1)$     &  $\zeta_3^2(1,1,1,\zeta_3,\zeta_3,\zeta_3,\zeta_3^2,\zeta_3^2,\zeta_3^2)$        &      $\zeta_3^2H_{9}^{(3)}$\\
\hline
$\mathbf{u_{32}} =(2,1,0,1)$     &  $\zeta_3^2(1,1,1,\zeta_3,\zeta_3,\zeta_3,\zeta_3^2,\zeta_3^2,\zeta_3^2)$        &      $\zeta_3^2H_{9}^{(3)}$\\
\hline
$\mathbf{u_{33}} =(0,2,0,1)$     &  $\zeta_3(1,1,1,\zeta_3,\zeta_3,\zeta_3,\zeta_3^2,\zeta_3^2,\zeta_3^2)$        &      $\zeta_3H_{9}^{(3)}$\\
\hline
$\mathbf{u_{34}} =(1,2,0,1)$     &  $\zeta_3(1,1,1,\zeta_3,\zeta_3,\zeta_3,\zeta_3^2,\zeta_3^2,\zeta_3^2)$        &      $\zeta_3H_{9}^{(3)}$\\
\hline
$\mathbf{u_{35}} =(2,2,0,1)$     &  $\zeta_3(1,1,1,\zeta_3,\zeta_3,\zeta_3,\zeta_3^2,\zeta_3^2,\zeta_3^2)$        &      $\zeta_3H_{9}^{(3)}$\\
\hline
$\mathbf{u_{36}} =(0,0,1,1)$     &  $(1,\zeta_3^2,\zeta_3,1,\zeta_3^2,\zeta_3,1,\zeta_3^2,\zeta_3)$        &      $H_{9}^{(2)}$\\
\hline
$\mathbf{u_{37}} =(1,0,1,1)$     &  $\zeta_3(1,\zeta_3^2,\zeta_3,1,\zeta_3^2,\zeta_3,1,\zeta_3^2,\zeta_3)$        &      $\zeta_3H_{9}^{(2)}$\\
\hline
$\mathbf{u_{38}} =(2,0,1,1)$     &  $\zeta_3^2(1,\zeta_3^2,\zeta_3,1,\zeta_3^2,\zeta_3,1,\zeta_3^2,\zeta_3)$        &      $\zeta_3^2H_{9}^{(2)}$\\
\hline
$\mathbf{u_{39}} =(0,1,1,1)$     &  $\zeta_3^2(1,\zeta_3^2,\zeta_3,1,\zeta_3^2,\zeta_3,1,\zeta_3^2,\zeta_3)$        &      $\zeta_3^2H_{9}^{(2)}$\\
\hline
$\mathbf{u_{40}} =(1,1,1,1)$     &  $(1,\zeta_3^2,\zeta_3,1,\zeta_3^2,\zeta_3,1,\zeta_3^2,\zeta_3)$        &      $H_{9}^{(2)}$\\
\hline
$\mathbf{u_{41}} =(2,1,1,1)$     &  $\zeta_3(1,\zeta_3^2,\zeta_3,1,\zeta_3^2,\zeta_3,1,\zeta_3^2,\zeta_3)$        &      $\zeta_3H_{9}^{(2)}$\\
\hline
$\mathbf{u_{42}} =(0,2,1,1)$     &  $\zeta_3(1,\zeta_3^2,\zeta_3,1,\zeta_3^2,\zeta_3,1,\zeta_3^2,\zeta_3)$        &      $\zeta_3H_{9}^{(2)}$\\
\hline
$\mathbf{u_{43}} =(1,2,1,1)$     &  $\zeta_3^2(1,\zeta_3^2,\zeta_3,1,\zeta_3^2,\zeta_3,1,\zeta_3^2,\zeta_3)$        &      $\zeta_3^2H_{9}^{(2)}$\\
\hline
$\mathbf{u_{44}} =(2,2,1,1)$     &  $(1,\zeta_3^2,\zeta_3,1,\zeta_3^2,\zeta_3,1,\zeta_3^2,\zeta_3)$        &      $H_{9}^{(2)}$\\
\hline
$\mathbf{u_{45}} =(0,0,2,1)$     &  $(1,\zeta_3,\zeta_3^2,\zeta_3^2,1,\zeta_3,\zeta_3,\zeta_3^2,1)$        &      $H_{9}^{(7)}$\\
\hline
$\mathbf{u_{46}} =(1,0,2,1)$     &  $\zeta_3^2(1,\zeta_3,\zeta_3^2,\zeta_3^2,1,\zeta_3,\zeta_3,\zeta_3^2,1)$        &      $\zeta_3^2H_{9}^{(7)}$\\
\hline
$\mathbf{u_{47}} =(2,0,2,1)$     &  $\zeta_3(1,\zeta_3,\zeta_3^2,\zeta_3^2,1,\zeta_3,\zeta_3,\zeta_3^2,1)$        &      $\zeta_3H_{9}^{(7)}$\\
\hline
$\mathbf{u_{48}} =(0,1,2,1)$     &  $\zeta_3^2(1,\zeta_3,\zeta_3^2,\zeta_3^2,1,\zeta_3,\zeta_3,\zeta_3^2,1)$        &      $\zeta_3^2H_{9}^{(7)}$\\
\hline
$\mathbf{u_{49}} =(1,1,2,1)$     &  $\zeta_3(1,\zeta_3,\zeta_3^2,\zeta_3^2,1,\zeta_3,\zeta_3,\zeta_3^2,1)$        &      $\zeta_3H_{9}^{(7)}$\\
\hline
$\mathbf{u_{50}} =(2,1,2,1)$     &  $(1,\zeta_3,\zeta_3^2,\zeta_3^2,1,\zeta_3,\zeta_3,\zeta_3^2,1)$        &      $H_{9}^{(7)}$\\
\hline
$\mathbf{u_{51}} =(0,2,2,1)$     &  $\zeta_3(1,\zeta_3,\zeta_3^2,\zeta_3^2,1,\zeta_3,\zeta_3,\zeta_3^2,1)$        &      $\zeta_3H_{9}^{(7)}$\\
\hline
$\mathbf{u_{52}} =(1,2,2,1)$     &  $(1,\zeta_3,\zeta_3^2,\zeta_3^2,1,\zeta_3,\zeta_3,\zeta_3^2,1)$        &      $H_{9}^{(7)}$\\
\hline
$\mathbf{u_{53}} =(2,2,2,1)$     &  $\zeta_3^2(1,\zeta_3,\zeta_3^2,\zeta_3^2,1,\zeta_3,\zeta_3,\zeta_3^2,1)$        &      $\zeta_3^2H_{9}^{(7)}$\\
\hline
$\mathbf{u_{54}} =(0,0,0,2)$     &  $(1,1,1,\zeta_3^2,\zeta_3^2,\zeta_3^2,\zeta_3,\zeta_3,\zeta_3^2)$        &      $H_{9}^{(6)}$\\
\hline
$\mathbf{u_{55}} =(1,0,0,2)$     &  $(1,1,1,\zeta_3^2,\zeta_3^2,\zeta_3^2,\zeta_3,\zeta_3,\zeta_3^2)$        &      $H_{9}^{(6)}$\\
\hline
$\mathbf{u_{56}} =(2,0,0,2)$     &  $(1,1,1,\zeta_3^2,\zeta_3^2,\zeta_3^2,\zeta_3,\zeta_3,\zeta_3^2)$        &      $H_{9}^{(6)}$\\
\hline
$\mathbf{u_{57}} =(0,1,0,2)$     &  $\zeta_3(1,1,1,\zeta_3^2,\zeta_3^2,\zeta_3^2,\zeta_3,\zeta_3,\zeta_3^2)$        &      $\zeta_3H_{9}^{(6)}$\\
\hline
$\mathbf{u_{58}} =(1,1,0,2)$     &  $\zeta_3(1,1,1,\zeta_3^2,\zeta_3^2,\zeta_3^2,\zeta_3,\zeta_3,\zeta_3^2)$        &      $\zeta_3H_{9}^{(6)}$\\
\hline
$\mathbf{u_{59}} =(2,1,0,2)$     &  $\zeta_3(1,1,1,\zeta_3^2,\zeta_3^2,\zeta_3^2,\zeta_3,\zeta_3,\zeta_3^2)$        &      $\zeta_3H_{9}^{(6)}$\\
\hline
$\mathbf{u_{60}} =(0,2,0,2)$     &  $\zeta_3^2(1,1,1,\zeta_3^2,\zeta_3^2,\zeta_3^2,\zeta_3,\zeta_3,\zeta_3^2)$        &      $\zeta_3^2H_{9}^{(6)}$\\
\hline
$\mathbf{u_{61}} =(1,2,0,2)$     &  $\zeta_3^2(1,1,1,\zeta_3^2,\zeta_3^2,\zeta_3^2,\zeta_3,\zeta_3,\zeta_3^2)$        &      $\zeta_3^2H_{9}^{(6)}$\\
\hline
$\mathbf{u_{62}} =(2,2,0,2)$     &  $\zeta_3^2(1,1,1,\zeta_3^2,\zeta_3^2,\zeta_3^2,\zeta_3,\zeta_3,\zeta_3^2)$        &      $\zeta_3^2H_{9}^{(6)}$\\
\hline
$\mathbf{u_{63}} =(0,0,1,2)$     &  $(1,\zeta_3^2,\zeta_3,\zeta_3,1,\zeta_3^2,\zeta_3^2,\zeta_3,1)$        &      $H_{9}^{(5)}$\\
\hline
$\mathbf{u_{64}} =(1,0,1,2)$     &  $\zeta_3(1,\zeta_3^2,\zeta_3,\zeta_3,1,\zeta_3^2,\zeta_3^2,\zeta_3,1)$        &      $\zeta_3H_{9}^{(5)}$\\
\hline
$\mathbf{u_{65}} =(2,0,1,2)$     &  $\zeta_3^2(1,\zeta_3^2,\zeta_3,\zeta_3,1,\zeta_3^2,\zeta_3^2,\zeta_3,1)$        &      $\zeta_3^2H_{9}^{(5)}$\\
\hline
$\mathbf{u_{66}} =(0,1,1,2)$     &  $\zeta_3(1,\zeta_3^2,\zeta_3,\zeta_3,1,\zeta_3^2,\zeta_3^2,\zeta_3,1)$        &      $\zeta_3H_{9}^{(5)}$\\
\hline
$\mathbf{u_{67}} =(1,1,1,2)$     &  $\zeta_3^2(1,\zeta_3^2,\zeta_3,\zeta_3,1,\zeta_3^2,\zeta_3^2,\zeta_3,1)$        &      $\zeta_3^2H_{9}^{(5)}$\\
\hline
$\mathbf{u_{68}} =(2,1,1,2)$     &  $(1,\zeta_3^2,\zeta_3,\zeta_3,1,\zeta_3^2,\zeta_3^2,\zeta_3,1)$        &      $H_{9}^{(5)}$\\
\hline
$\mathbf{u_{69}} =(0,2,1,2)$     &  $\zeta_3^2(1,\zeta_3^2,\zeta_3,\zeta_3,1,\zeta_3^2,\zeta_3^2,\zeta_3,1)$        &      $\zeta_3^2H_{9}^{(5)}$\\
\hline
$\mathbf{u_{70}} =(1,2,1,2)$     &  $(1,\zeta_3^2,\zeta_3,\zeta_3,1,\zeta_3^2,\zeta_3^2,\zeta_3,1)$        &      $H_{9}^{(5)}$\\
\hline
$\mathbf{u_{71}} =(2,2,1,2)$     &  $\zeta_3(1,\zeta_3^2,\zeta_3,\zeta_3,1,\zeta_3^2,\zeta_3^2,\zeta_3,1)$        &      $\zeta_3H_{9}^{(5)}$\\
\hline
$\mathbf{u_{72}} =(0,0,2,2)$     &  $(1,\zeta_3,\zeta_3^2,1,\zeta_3,\zeta_3^2,1,\zeta_3,\zeta_3^2)$        &      $H_{9}^{(1)}$\\
\hline
$\mathbf{u_{73}} =(1,0,2,2)$     &  $\zeta_3^2(1,\zeta_3,\zeta_3^2,1,\zeta_3,\zeta_3^2,1,\zeta_3,\zeta_3^2)$        &      $\zeta_3^2H_{9}^{(1)}$\\
\hline
$\mathbf{u_{74}} =(2,0,2,2)$     &  $\zeta_3(1,\zeta_3,\zeta_3^2,1,\zeta_3,\zeta_3^2,1,\zeta_3,\zeta_3^2)$        &      $\zeta_3H_{9}^{(1)}$\\
\hline
$\mathbf{u_{75}} =(0,1,2,2)$     &  $\zeta_3(1,\zeta_3,\zeta_3^2,1,\zeta_3,\zeta_3^2,1,\zeta_3,\zeta_3^2)$        &      $\zeta_3H_{9}^{(1)}$\\
\hline
$\mathbf{u_{76}} =(1,1,2,2)$     &  $(1,\zeta_3,\zeta_3^2,1,\zeta_3,\zeta_3^2,1,\zeta_3,\zeta_3^2)$        &          $H_{9}^{(1)}$\\
\hline
$\mathbf{u_{77}} =(2,1,2,2)$     &  $\zeta_3^2(1,\zeta_3,\zeta_3^2,1,\zeta_3,\zeta_3^2,1,\zeta_3,\zeta_3^2)$        &      $\zeta_3^2H_{9}^{(1)}$\\
\hline
$\mathbf{u_{78}} =(0,2,2,2)$     &  $\zeta_3^2(1,\zeta_3,\zeta_3^2,1,\zeta_3,\zeta_3^2,1,\zeta_3,\zeta_3^2)$        &      $\zeta_3^2H_{9}^{(1)}$\\
\hline
$\mathbf{u_{79}} =(1,2,2,2)$     &  $\zeta_3(1,\zeta_3,\zeta_3^2,1,\zeta_3,\zeta_3^2,1,\zeta_3,\zeta_3^2)$        &          $\zeta_3H_{9}^{(1)}$\\
\hline
$\mathbf{u_{80}} =(2,2,2,2)$     &  $(1,\zeta_3,\zeta_3^2,1,\zeta_3,\zeta_3^2,1,\zeta_3,\zeta_3^2)$        &      $H_{9}^{(1)}$\\
\hline
\end{longtable}
\end{appendices}

\begin{appendices}
\section{$\mathcal{H}_{\textbf{a}}(\textbf{u})$ for $f$ in Example 4.2}
\begin{longtable}{|c|c|c|c|c|c|}
\caption{Vectors $\mathcal{H}_{\textbf{a}}(\textbf{u})$ for all $\mathbf{u} \in \Z_3^4$} \label{table2}\\
\hline
$\mathbf{u} \in \Z_3^4$ 　&　$\mathcal{H}_{\textbf{a}}(\textbf{u})$  &  $\mathbf{u} \in \Z_3^4$ 　& 　$\mathcal{H}_{\textbf{a}}(\textbf{u}) $    &  $\mathbf{u} \in \Z_3^4$ 　& 　$\mathcal{H}_{\textbf{a}}(\textbf{u}) $\\
\hline

\hline
$(0,0,0,0)$   &    $H_{9}^{(1)}$   &   $(0,0,0,1)$   &    $H_{9}^{(7)}$  &   $(0,0,0,2)$   &    $H_{9}^{(4)}$\\
\hline
$(1,0,0,0)$   &    $H_{9}^{(1)}$   &   $(1,0,0,1)$   &    $H_{9}^{(7)}$  &   $(1,0,0,2)$   &    $H_{9}^{(4)}$\\
\hline
$(2,0,0,0)$   &    $H_{9}^{(1)}$   &   $(2,0,0,1)$   &    $H_{9}^{(7)}$  &   $(2,0,0,2)$   &    $H_{9}^{(4)}$\\
\hline
$(0,1,0,0)$   &    $H_{9}^{(1)}$   &   $(0,1,0,1)$   &    $\zeta_3H_{9}^{(7)}$  &   $(0,1,0,2)$   &    $\zeta_3^2H_{9}^{(4)}$\\
\hline
$(1,1,0,0)$   &    $H_{9}^{(1)}$   &   $(1,1,0,1)$   &    $\zeta_3H_{9}^{(7)}$  &   $(1,1,0,2)$   &    $\zeta_3^2H_{9}^{(4)}$\\
\hline
$(2,1,0,0)$   &    $H_{9}^{(1)}$   &   $(2,1,0,1)$   &    $\zeta_3H_{9}^{(7)}$  &   $(2,1,0,2)$   &    $\zeta_3^2H_{9}^{(4)}$\\
\hline
$(0,2,0,0)$   &    $H_{9}^{(1)}$   &   $(0,2,0,1)$   &    $\zeta_3^2H_{9}^{(7)}$  &   $(0,2,0,2)$   &    $\zeta_3H_{9}^{(4)}$\\
\hline
$(1,2,0,0)$   &    $H_{9}^{(1)}$   &   $(1,2,0,1)$   &    $\zeta_3^2H_{9}^{(7)}$  &   $(1,2,0,2)$   &    $\zeta_3H_{9}^{(4)}$\\
\hline
$(2,2,0,0)$   &    $H_{9}^{(1)}$   &   $(2,2,0,1)$   &    $\zeta_3^2H_{9}^{(7)}$  &   $(2,2,0,2)$   &    $\zeta_3H_{9}^{(4)}$\\
\hline

$(0,0,1,0)$   &    $H_{9}^{(7)}$   &   $(0,0,1,1)$   &    $H_{9}^{(4)}$  &   $(0,0,1,2)$   &    $H_{9}^{(1)}$\\
\hline
$(1,0,1,0)$   &    $\zeta_3^2H_{9}^{(7)}$   &   $(1,0,1,1)$   &    $\zeta_3^2H_{9}^{(4)}$  &   $(1,0,1,2)$   &    $\zeta_3^2H_{9}^{(1)}$\\
\hline
$(2,0,1,0)$   &    $\zeta_3H_{9}^{(7)}$   &   $(2,0,1,1)$   &    $\zeta_3H_{9}^{(4)}$  &   $(2,0,1,2)$   &    $\zeta_3H_{9}^{(1)}$\\
\hline
$(0,1,1,0)$   &    $H_{9}^{(7)}$   &   $(0,1,1,1)$   &    $\zeta_3H_{9}^{(4)}$  &   $(0,1,1,2)$   &    $\zeta_3^2H_{9}^{(1)}$\\
\hline
$(1,1,1,0)$   &    $\zeta_3^2H_{9}^{(7)}$   &   $(1,1,1,1)$   &    $H_{9}^{(4)}$  &   $(1,1,1,2)$   &    $\zeta_3H_{9}^{(1)}$\\
\hline
$(2,1,1,0)$   &    $\zeta_3H_{9}^{(7)}$   &   $(2,1,1,1)$   &    $\zeta_3^2H_{9}^{(4)}$  &   $(2,1,1,2)$   &    $H_{9}^{(1)}$\\
\hline
$(0,2,1,0)$   &    $H_{9}^{(7)}$   &   $(0,2,1,1)$   &    $\zeta_3^2H_{9}^{(4)}$  &   $(0,2,1,2)$   &    $\zeta_3H_{9}^{(1)}$\\
\hline
$(1,2,1,0)$   &    $\zeta_3^2H_{9}^{(7)}$   &   $(1,2,1,1)$   &    $\zeta_3H_{9}^{(4)}$  &   $(1,2,1,2)$   &    $H_{9}^{(1)}$\\
\hline
$(2,2,1,0)$   &    $\zeta_3H_{9}^{(7)}$   &   $(2,2,1,1)$   &    $H_{9}^{(4)}$  &   $(2,2,1,2)$   &    $\zeta_3^2H_{9}^{(1)}$\\
\hline

$(0,0,2,0)$   &    $H_{9}^{(4)}$   &   $(0,0,2,1)$   &    $H_{9}^{(1)}$  &   $(0,0,2,2)$   &    $H_{9}^{(7)}$\\
\hline
$(1,0,2,0)$   &    $\zeta_3H_{9}^{(4)}$   &   $(1,0,2,1)$   &    $\zeta_3H_{9}^{(1)}$  &   $(1,0,2,2)$   &    $\zeta_3H_{9}^{(7)}$\\
\hline
$(2,0,2,0)$   &    $\zeta_3^2H_{9}^{(4)}$   &   $(2,0,2,1)$   &    $\zeta_3^2H_{9}^{(1)}$  &   $(2,0,2,2)$   &    $\zeta_3^2H_{9}^{(7)}$\\
\hline
$(0,1,2,0)$   &    $H_{9}^{(4)}$   &   $(0,1,2,1)$   &    $\zeta_3H_{9}^{(1)}$  &   $(0,1,2,2)$   &    $\zeta_3^2H_{9}^{(7)}$\\
\hline
$(1,1,2,0)$   &    $\zeta_3H_{9}^{(4)}$   &   $(1,1,2,1)$   &    $\zeta_3^2H_{9}^{(1)}$  &   $(1,1,2,2)$   &    $H_{9}^{(7)}$\\
\hline
$(2,1,2,0)$   &    $\zeta_3^2H_{9}^{(4)}$   &   $(2,1,2,1)$   &    $H_{9}^{(1)}$  &   $(2,1,2,2)$   &    $\zeta_3H_{9}^{(7)}$\\
\hline
$(0,2,2,0)$   &    $H_{9}^{(4)}$   &   $(0,2,2,1)$   &    $\zeta_3^2H_{9}^{(1)}$  &   $(0,2,2,2)$   &    $\zeta_3H_{9}^{(7)}$\\
\hline
$(1,2,2,0)$   &    $\zeta_3H_{9}^{(4)}$   &   $(1,2,2,1)$   &    $H_{9}^{(1)}$  &   $(1,2,2,2)$   &    $\zeta_3^2H_{9}^{(7)}$\\
\hline
$(2,2,2,0)$   &    $\zeta_3^2H_{9}^{(4)}$   &   $(2,2,2,1)$   &    $\zeta_3H_{9}^{(1)}$  &   $(2,2,2,2)$   &    $H_{9}^{(7)}$\\
\hline
\end{longtable}
\end{appendices}

\end{document}